\documentclass{amsart}
\usepackage{amssymb,amsmath}

\usepackage{float}
\usepackage{subfig,epsfig}
\usepackage{graphicx}

\newtheorem{theorem}{Theorem}[section]
\newtheorem{lemma}[theorem]{Lemma}

\newtheorem{corollary}[theorem]{Corollary}
\theoremstyle{definition}
\newtheorem{definition}[theorem]{Definition}

\theoremstyle{remark}
\newtheorem{remark}[theorem]{Remark}

\numberwithin{equation}{section}

\begin{document}
\makeatletter
\@namedef{subjclassname@2020}{%
\textup{2020} Mathematics Subject Classification}
\makeatother

\author{Fengjiang Li}
\address[Fengjiang Li]
   {1. Mathematical Science Research Center\\ Chongqing University of Technology\\ Chongqing\\ 400054. 
	\newline 
	2. School of Mathematical Sciences\\ East China Normal University\\
	Shanghai\\ 200241}

\email{lianyisky@163.com}


\title[Gradient Steady Ricci Solitons with Harmonic Weyl Curvature]
{Rigidity of Complete Gradient Steady Ricci Solitons with Harmonic Weyl Curvature}


\begin{abstract}
Our main aim in this paper is to investigate the rigidity of complete noncompact gradient steady Ricci solitons with harmonic Weyl tensor. More precisely, we prove that an $n$-dimensional ($n\geq 5$) complete noncompact gradient steady Ricci soliton with harmonic Weyl tensor and multiply warped product metric is either Ricci flat or isometric to the Bryant soliton up to scaling. 
Meanwhile, for $n\ge 5$, we provide a local structure theorem for $n$-dimensional connected (not necessarily complete) gradient Ricci solitons with harmonic Weyl curvature and multiply warped product metric.
\end{abstract}

\maketitle

\section{Introduction}
An $n$-dimensional Riemannian manifold $(M^{n},g)$ is called a gradient Ricci soliton if there exists a smooth function $f$ on $M$ such that the Ricci tensor satisfies the following equation
\begin{equation}\label{1.1}
Ric+Hess(f)=\rho g
\end{equation}
for some constant $\rho$, where $Ric$ is the Ricci tensor of $g$ and $Hess(f)$ denotes the Hessian of the potential function $f$. 
The Ricci soliton is said to be shrinking, steady, or expanding accordingly as $\rho$ is positive, zero, or negative, respectively. When the potential function $f$ is constant, the gradient Ricci soliton is simply an Einstein manifold and is said to be trivial. Ricci solitons generate self-similar solutions of the Ricci flow, 
and they play a fundamental role in the formation of singularities of the flow (see \cite{Cao1} for a nice overview).
Hence the classification of gradient Ricci solitons has been a very interesting problem. 

In the shrinking case, classification results for gradient Ricci solitons have been obtained by many authors under various curvature conditions on the Weyl tensor, e.g., \cite{Iv3, P, ELM, Naber, NW, CCZ, PW2, CWZ, Z,  FG, MS, CC2, CW,WWW} etc. 
In particular, Fern\'{a}ndez-L\'{o}pez and Garc\'{i}a-R\'{i}o \cite{FG} together with Munteanu and Sesum \cite{MS} proved that any $n$-dimensional complete gradient shrinker with harmonic Weyl tensor is rigid, i.e., it is isometric to a (finite) quotient of $N \times \mathbb{R}^k$, the product soliton of an Einstein manifold $N$ of positive scalar curvature with the Gaussian soliton $\mathbb{R}^k$. 

On the other hand, it is well-known that compact gradient steady solitons are necessarily Ricci flat.
In dimension $n=2$, the only complete noncompact gradient steady Ricci soliton with positive curvature is Hamilton's cigar soliton $\Sigma^2$, see Hamilton \cite{Ha95F}. In dimension three, known examples are given by $\mathbb R^{3}$, $\Sigma^{2}\times \mathbb R$, and the rotationally symmetric Bryant soliton \cite{Bryant}. In \cite{Br}, Brendle showed that the Bryant soliton is the only complete noncompact, nonflat, $\kappa$-noncollapsed, gradient steady Ricci soliton,  proving a conjecture by Perelman \cite{P}. For $n\ge 4$, such a uniqueness result is not expected to hold, and it is desirable to find geometrically interesting conditions under which the uniqueness would hold. Indeed, in the K\"ahler case, Cao \cite{Cao2} constructed a complete gradient steady K\"ahler-Ricci soliton on $\mathbb{C}^m$, for $m\geq 2$, with positive sectional curvature and $U(m)$ symmetry. 

In \cite{CC1}, for $n\ge 3$, Cao-Chen showed that a complete noncompact $n$-dimensional locally conformally flat gradient steady Ricci soliton is either flat or isometric to the Bryant soliton up to scaling; see also \cite{CM} for an independent proof for $n\ge 4$.  Moreover, an important covariant $3$-tensor $D$ for gradient Ricci solitons was introduced in \cite{CC1, CC2} (see also Section 2 for the definition of $D$-tensor). 
Classification results have been obtained in \cite{CCCMM} for Bach flat steady solitons in dimension $n\geq 4$ under some conditions. In particular,  it follows that Bach flatness implies local conformal flatness under positive Ricci curvature assumption.  However, not much was known about the rigidity of complete gradient steady Ricci solitons with harmonic curvature. In fact, it is quite natural to ask the following: 

\medskip
\noindent \textbf{Main Question}.
\emph{Is it true that any $n$-dimensional $(n\geq 4)$ steady gradient Ricci soliton $(M, g, f)$ with harmonic Weyl curvature is either Ricci flat or isometric to the Bryant soliton up to scaling?}

\medskip
Recently, Kim \cite{Kim} has provided a positive answer to the above question for $n=4$. 
In fact, Kim produced a very nice local description of such Ricci soliton metrics and their potential functions. His method of proof was motivated by the work of Cao-Chen \cite{CC1, CC2}, and also based on Derdzi\'{n}ski's study on Codazzi tensors \cite{De} (the harmonicity of the Weyl tensor is equivalent to the Schouten tensor being Codazzi).  
Combining with the gradient Ricci soliton condition, Kim managed to analyze in detail the situation when the Ricci tensor has two and three distinct eigenvalues. However, difficulties arise in the higher dimensional case. Indeed, as the dimension increases, so do the numbers and multiplicities of distinct Ricci-eigenvalues and the situation becomes more subtle.

\medskip
On the other hand, the condition harmonic Weyl curvature is equivalent to the Schouten tensor being a Codazzi tensor. 
Derdzi\'{n}ski \cite{De} described the following: for a Codazzi tensor $A$ and a point $x$ in $M$, 
let $E_A(x)$ be the number of distinct eigenvalues of $A_x$,
and set 
\[
M_A = \{  x \in M \ | E_A {\rm \ is \ constant \ in \ a \ neighborhood \ of \ } x \}.
\] 
Then $M_A$ is an open dense subset of $M$;
in each connected component of $M_A$, 
the eigenvalues are well-defined and differentiable functions.
As a gradient Ricci soliton, 
$(M,g,f)$ is real analytic in harmonic coordinates; see \cite{Iv, HPW}. 
Then if $f$ is not constant, $\{\nabla f \neq 0\}$ is open and dense in $M$ and $\nabla f $ is an eigenvector field of the Ricci tensor (see \ref{lemma3.1}). 
Hence for each point $p\in M_A \cap \{ \nabla f \neq 0  \}$, there exists a neighborhood $U$ of $p$, 
such that the number of distinct eigenvalues of the Ricci tensor is constant on $U$,
and we assume that the multiplicities of $m$ distinct Ricci eigenvalues are $r_{1}, r_{2}, \cdots, r_{m}$, respectively, 
except for the one, which is the eigenvalue with respect to the eigenvector $\nabla f$,
where $1+r_{1}+r_{2}+ \cdots+ r_{m}=n $.

\begin{definition}
A gradient Ricci soliton metric $g$ is called multiply warped product metric of eigenspaces with the Ricci tensor,
if for each point $p \in M_A \cap \{ \nabla f \neq 0  \}$, there exists a neighborhood $U$ of $p$ 
such that 
\[
U = L\times  _{h_1}L_1 \cdots \times _{h_l}L_l \times _{h_{l+1}}N_{l+1}\times\cdots \times _{h_{m}}N_{m}
\]
and the metric $g$, when restricted to $U$, is a multiply warped product metric of the form
\begin{equation*}
	g= ds^2 + h^2_1(s)  dt_1^2 + \cdots 
	+ h^2_l(s) dt_l^2+ h^2_{l+1}(s) \tilde{g}_{l+1}+ \cdots 
	+h^2_{m}(s) \tilde{g}_{m},
\end{equation*}
where $h_j(s)$ are smooth positive functions for $1\leq j \leq m$,
$dim ~L_\nu$=1 for $1\leq \nu \leq l$, and 
$(N_\mu, \tilde{g}_{\mu})$ is an $r_\mu$-dimensional integral submanifold of the eigenspaces corresponding to the Ricci-eigenvalue for $l+1\leq \mu \leq m$.	
\end{definition}

In this paper, motivated by Kim's work \cite{Kim}, we study $n$-dimensional ($n\geq 5$) complete gradient Ricci solitons with harmonic Weyl curvature and multiply warped product metric of eigenspaces with the Ricci tensor. Our first main result is the following classification result. 

\begin{theorem} \label{complete}
	Let $(M^n, g, f)$, $n\geq 5$,  be a complete $n$-dimensional gradient Ricci soliton with harmonic Weyl curvature and multiply warped product metric of eigenspaces  with the Ricci tensor. Then it is one of the following types:
		
	\smallskip	
	{\rm (i)} 
	$(M^n, g)$ is isometric to a quotient of $\mathbb{R}^{r}\times N^{n-r}$ ($2\leq r\leq n-2$), where $N^{n-r}$ is an $(n-r)$-dimensional Einstein manifold with the Einstein constant $\rho\neq 0$.
	Also the potential function is given by $f = \frac{\rho}{2} |x|^2$ modulo a constant on the Euclidean factor.  
	
	\smallskip	
	{\rm (ii)} $(M^n, g)$ has the vanishing covariant $3$-tensor $D$ of Cao-Chen \cite{CC1, CC2}. 
\end{theorem}

Note that, as mentioned before, the analog of Theorem 1.1 in the shrinking case was already known \cite{FG, MS}. In addition, Theorem \ref{complete} (ii) includes the case when $(M^n, g)$ is Einstein: if $f$ is a constant function then the vanishing of $D$-tensor follows from its definition, and the case of nonconstant $f$ follows from Theorem \ref{same} or from, e.g., the work of Cheeger-Colding \cite{CC} on warped products and Hessians.

Meanwhile, Cao-Chen \cite{CC2}  showed that any $n$-dimensional ($n\geq 4$) gradient Ricci soliton with vanishing $D$-tensor has harmonic Weyl curvature. Therefore, together with Theorem \ref{complete},  an $n$-dimensional ($n\geq 5$) gradient {\it steady} Ricci soliton has harmonic Weyl curvature and multiply warped product metric if and only if the $D$-tensor vanishes.
In particular, as a consequence of Theorem \ref{complete} and the very recent work of Cao-Yu \cite{CY}, we have the following partial answer to Main Question for all $n\ge 5$.  

\begin{theorem} \label{steady}
Let $(M^n, g, f)$, $n\geq 5$, be a complete noncompact gradient steady Ricci soliton with harmonic Weyl curvature  and multiply warped product metric of eigenspaces  with the Ricci tensor. Then it is either Ricci flat or isometric to the Bryant soliton up to scaling.
\end{theorem}

On the other hand, the {\it expanding} solitons are less rigid, and various works have been done recently; see, e.g., \cite{PW, CD, Cho} and the references therein. As another consequence of Theorem \ref{complete} and the work of Cao-Yu \cite{CY} for complete $D$-flat expanding solitons, we have the following classification result.

\begin{theorem} \label{expand}
Let $(M, g,f)$, $n\geq 5$, be a complete expanding gradient Ricci soliton with harmonic Weyl curvature and multiply warped product metric of eigenspaces  with the Ricci tensor.
Then it is one of the following types.
	
	\smallskip	
	{\rm (i)} $g$ is an Einstein metric with a constant potential function $f$.
	
	\smallskip	
	{\rm (ii)}
	$(M^n, g)$ is isometric to a quotient of 
	$ \mathbb{R}^{r} \times {N}^{n-r}$,
	where $2 \leq r\leq n-2$, $ \mathbb{R}^{r}$ is the Gaussian expanding soliton, and ${N}^{n-r}$ is an $(n-r)$-dimensional Einstein manifold with the Einstein constant $\rho < 0$.	
	
	\smallskip
	{\rm (iii)} $(M^n, g)$ is rotationally  symmetric and a quotient of an expanding soliton of the form
	\[
	\left( [0,\, \infty),\, ds^2 \right) \times \,_h\left(\mathbb{S}^{n-1}, \bar{g}_0\right) 
	\]
    where $\bar{g}_0$ is the round metric on $\mathbb{S}^{n-1}$.

	\smallskip
{\rm (iv)} $(M^n, g)$ is a quotient of some warped product expanding Ricci soliton of the form
	\[
\left(\mathbb{R} ,\, ds^2 \right) \times\, _h\left(N^{n-1}, \bar{g} \right) 
\]
where $\left(N^{n-1}, \bar{g} \right) $ is an Einstein manifold of negative scalar curvature.
\end{theorem}

Inspired by the work of Cao-Chen \cite{CC1, CC2} and the work of Kim \cite{Kim}, we in fact derived a local description of Ricci soliton potential functions under the assumption of harmonic Weyl curvature. By using the method of exterior differential and moving frames, we first obtain the integrability condition of harmonicity and give a local structure of the soliton metric with a multiply warped product. Then we divide the discussion into three cases, according to the numbers and multiplicities of distinct Ricci-eigenvalues. As a result, we obtain the following local classification. 

\begin{theorem} \label{local}
	Let $(M^n, g, f)$, $(n\geq 5)$, be an $n$-dimensional (not necessarily complete) connected gradient Ricci soliton with harmonic Weyl curvature and multiply warped product metric of eigenspaces with the Ricci tensor. Then it is one of the following four types.
	
	\smallskip
	{\rm (i)} $(M^n, g)$ is Einstein, and $f$ is a constant function.
	
	\smallskip
	{\rm (ii)} For each point $p \in M$, there exists a neighborhood $U$ of $p$ such that $(U, g)$ is isometric to a domain in the Riemannian product $\mathbb{R}^{r}\times N^{n-r}$ with $g= ds^2 + s^2d\mathbb{S}^2_{r-1}+\tilde{g} $,
	where $2 \leq r\leq n-2$ and $\left(N^{n-r}, \tilde{g}\right)$ is Einstein with the Einstein constant $\rho\neq 0$.
	The potential function is given by $f = \frac{\rho}{2} |x|^2$ modulo a constant on the Euclidean factor.
	
	\smallskip
	{\rm (iii)} For each point $p \in M$, there exists a neighborhood $U$ of $p$ 
	such that $(U,g)$ is isometric to a domain in 
	$\mathbb{R^+}\times\mathbb{R} \times N^{n-2}$ 
	with 
	$g=ds^2 + s^{\frac{2(n-3)}{n-1}} dt^2+ s^{\frac{4}{n-1}} \tilde{g}$,  
	where $\left( N^{n-2}, \tilde{g}\right) $ is Ricci flat. 
	Also, $n\neq5$, $\rho=0$ and $f=\frac{2(n-3)}{n-1} \log (s)$ modulo a constant.
	
	\smallskip
	{\rm (iv)} For each point $p \in M$, 
	there exists a neighborhood $U$ of $p$ 
	such that $(U,g)$ is isometric to a domain in $\mathbb{R} \times N^{n-1}$ 
	with $g= ds^2 + h^2(s) \tilde{g},$
	where $\tilde{g}$ is an Einstein metric  
	on some $(n-1)$-manifold $N^{n-1}$.
	Moreover, the covariant $3$-tensor $D$ of Cao-Chen \cite{CC1, CC2} vanishes.
\end{theorem}

We point out that the incomplete steady gradient soliton in Theorem \ref{local} {\rm (iii)} has negative scalar curvature, which is in contrast to the fact that complete steady gradient solitons should have nonnegative scalar curvature \cite{Ch}. Thus, Theorem \ref{complete} follows immediately from Theorem \ref{local}.
\begin{remark} 
	The proof of Theorem \ref{local} has been further extended to treat vacuum static spaces and CPE metrics with harmonic curvature \cite {F-M-1, CL}, as well as quasi-Einstein manifolds with harmonic Weyl curvature \cite{CLS},  under the assumption of multiply warped product metric of eigenspaces  with the Ricci tensor.
\end{remark} 

Next, we discuss some applications. 
In \cite{CMM}, the authors obtained some results on gradient Ricci solitons with certain vanishing conditions on the Weyl tensor. More precisely, assuming ${\rm div}^4(W)=0$ and under certain Ricci curvature assumptions,
they showed that the Ricci soliton has harmonic Weyl curvature. 
Combining this fact with Theorem \ref{steady}, Theorem \ref{expand} and the work of \cite{Kim} for $n=4$, we have the following two corollaries. 

\begin{corollary}\label{cor1.4}
	Let $(M^n,g,f)$, $n\geq 4$, be a complete gradient steady Ricci soliton
	with positive Ricci curvature such that the scalar curvature attains its maximum at some point $p_0\in M$. 
	If the soliton is of multiply warped product metric and $\operatorname{div}^4(W)=0$, then $M$ is either Ricci flat or isometric to the Bryant soliton.
\end{corollary}

\begin{corollary}\label{cor1.6}
	Let $(M^n,g)$, $n\geq 4$, be a complete gradient expanding Ricci soliton with nonnegative Ricci curvature. 
	If the soliton is of multiply warped product metric and $\operatorname{div}^4(W)=0$, then $M$ is one of the following: 
	
	\smallskip
	{\rm (i)} $g$ is an Einstein metric with $f$ a constant function.
	
	\smallskip	
	{\rm (ii)} $(M^n, g)$ is isometric to a quotient of 
	$ \mathbb{R}^{r} \times {N}^{n-r}$,
	where $2 \leq r\leq n-2$, $ \mathbb{R}^{r}$ is the Gaussian expanding soliton,
	and ${N}^{n-r}$ is an $(n-r)$-dimensional Einstein
	manifold with the Einstein constant $\rho < 0$.	
	
	\smallskip	
	{\rm (iii)} $(M^n, g)$ is rotationally symmetric and a quotient of an expanding soliton of the form
	\[
	\left( [0,\, \infty),\, ds^2 \right) \times \,_h\left(\mathbb{S}^{n-1}, \bar{g}_0 \right) 
	\]
	where $\bar{g}_0$ is the round metric on $\mathbb{S}^{n-1}$.
	
	\smallskip
	{\rm (iv)} $(M^n, g)$ is a quotient of some warped product expanding Ricci soliton of the form
	\[
	\left(\mathbb{R} ,\, ds^2 \right) \times \, _h\left(N^{n-1}, \bar{g} \right) 
	\]
	where $\left(N^{n-1}, \bar{g} \right) $ is an Einstein manifold of negative scalar curvature.
\end{corollary}
	
	
	

This paper is organized as follows. 
In Section 2, we give some formulas and notations for Riemannian manifolds
and Ricci solitons by using the method of moving frames.
In Section 3, we derive the integrability conditions (ODEs) for a gradient Ricci soliton with harmonic Weyl tensor
when the local structure of the metric is a multiply warped product.
In Section 4-6, we divide our discussion into three cases according to the numbers and multiplicities of distinct
Ricci-eigenvalues $\lambda_i$, $i=1,2,\cdots, n$.
Here, we denote $\lambda_1$ as the Ricci-eigenvalue with respect to the gradient vector $\nabla f$ of the potential function. 
Concretely, in Section 4, we study the case that
there are at least three mutually different values in the eigenvalues 
$\lambda_2, \cdots, \lambda_n$,
but it turns out that this case cannot occur.
In Section 5, we analyze the case that 
there are exactly two distinct Ricci values in the eigenvalues $\lambda_2, \cdots, \lambda_n$.
Then two subcases are divided according to 
whether one of the two distinct Ricci-eigenvalues is of single multiplicity or not.
Types {\rm (ii)} and {\rm (iii)} of Theorem \ref{local} come from this part.
In Section 6, we treat the remaining case that all $\lambda_2,~\lambda_3,~\cdots,~\lambda_n$ are equal,
for which the $D$-tensor must vanish. 
In the last section, we summarize and prove the stated theorems.

\section{Preliminaries}
In this section, we first recall some formulae and notations 
for Riemannian manifolds by using the method of moving frames.
Then we give some fundamental formulae of Ricci solitons.

\subsection{Some notations for Riemannian manifolds.}
Let $M^{n}(n \geq 3)$ be an
$n$-dimensional Riemannian manifold,
$E_{1}, \cdots, E_{n}$ be a local orthonormal frame
fields on $M^{n}$, and $\omega_{1}, \cdots, \omega_{n}$
be their dual 1-forms. In this paper we make the following
conventions on the range of indices:
\[
1\leq i,j,k,\cdots\leq n
\]
and agree that repeated indices are summed over the respective ranges. 
Then we can write the structure equations of $M^{n}$ as follows:
\begin{equation}\label{2.1}
d\omega_{i}=\omega_{j}\wedge\omega_{ji}
\quad {\rm and }\quad \omega_{ij}+\omega_{ji}=0;
\end{equation}
\begin{equation}\label{2.2}
-\frac{1}{2}R_{ijkl}\omega_{k}\wedge\omega_{l}=
d\omega_{ij}-\omega_{ik}\wedge\omega_{kj}
\quad {\rm and }\quad R_{ijkl}=-R_{jikl},
\end{equation}
where $d$ is the exterior differential operator on $M$,
$\omega_{ij}$ is the Levi-Civita connection form
and $R_{ijkl}$ is the Riemannian curvature tensor of $M$.
It is known that the Riemannian curvature tensor satisfies
the following identities:
\begin{equation}\label{2.3}
R_{ijkl}=-R_{ijlk}, \quad R_{ijkl}=R_{klij}
\quad {\rm and }\quad R_{ijkl}+R_{iklj}+R_{iljk}=0.
\end{equation}
The Ricci tensor $R_{ij}$ and scalar curvature $R$ are defined respectively by
\begin{equation}\label{2.4}
R_{ij}:=\sum\limits_{k}R_{ikjk} \quad {\rm and }\quad R=\sum\limits_{i}R_{ii}.
\end{equation}
Let $f$ be a smooth function on $M^{n}$, we define the
covariant derivatives $f_{i}$, $f_{i,j}$ and $f_{i,jk}$ as follows:
\begin{equation}\label{2.5}
f_{i}\omega_{i}:=df,\quad  f_{i,j}\omega_{j}:=df_{i}+f_{j}\omega_{ji},
\end{equation}
and
\begin{equation}\label{2.6}
f_{i,jk}\omega_{k}:=df_{i,j}+f_{k,j}\omega_{ki}+f_{i,k}\omega_{kj}.
\end{equation}
We know that
\begin{equation}\label{2.7}
f_{i,j}=f_{j,i} \quad {\rm and }\quad f_{i,jk}-f_{i,kj}=f_{l}R_{lijk}.
\end{equation}
The gradient, Hessian and Laplacian of $f$ are defined by the following formulae:
\begin{equation}\label{2.8}
\nabla f:= f_{i}E_{i}, \quad  Hess(f):=f_{i,j}\omega_{i}\otimes \omega_{j} \quad {\rm and }\quad 
\Delta f:=\sum\limits_{i}f_{i,i}.
\end{equation}
The covariant derivatives of tensors $R_{ij}$ and $R_{ijkl}$
are defined by the following formulae:
\begin{equation}\label{2.9}
R_{ij,k}\omega_{k}:=dR_{ij}+R_{kj}\omega_{ki}+R_{ik}\omega_{kj}
\end{equation}
and
\begin{equation}\label{2.10}
R_{ijkl,m}\omega_{m}:=dR_{ijkl}+R_{mjkl}\omega_{mi}+
R_{imkl}\omega_{mj}+R_{ijml}\omega_{mk}+R_{ijkm}\omega_{ml}.
\end{equation}
By exterior differentiation of \eqref{2.2}, one can get the second Bianchi identity
\begin{equation}\label{2.11}
R_{ijkl,m}+R_{ijlm,k}+R_{ijmk,l}=0.
\end{equation}
From \eqref{2.4}, \eqref{2.10} and \eqref{2.11}, we have
\begin{equation}\label{2.12}
R_{ij,k}-R_{ik,j}=-\sum\limits_{l}R_{lijk,l},
\end{equation}
and so
\begin{equation}\label{2.13}
\sum\limits_{j}R_{ji,j}=\frac{1}{2}R_{i}.
\end{equation}

We define the Schouten tensor as
$A=A_{ij}\omega_{i}\otimes \omega_{j},$
where
\begin{equation}\label{2.14}
A_{ij}:=R_{ij}-\frac{1}{2(n-1)}R \delta_{ij},
\end{equation}
then $A_{ij}=A_{ji}$. The tensor
\begin{equation}\label{2.15}
W_{ijkl}:=R_{ijkl}-\frac{1}{n-2}
(A_{ik}\delta_{jl}+A_{jl}\delta_{ik}-
A_{il}\delta_{jk}-A_{jk}\delta_{il})
\end{equation}
is called the Weyl conformal curvature tensor which does not change under the conformal
transformation of the metric. Moreover, as it can be easily seen by the formula above,
$W$ is totally trace-free. 
In dimension three, $W$ is identically zero on every Riemannian manifold, whereas,
when $n\geq 4$, the vanishing of the Weyl tensor is
is equivalent to the locally conformal flatness of $(M^n,g)$. We also recall that in dimension $n=3$, 
$(M, g)$ is locally conformally flat if and only if the Cotton tensor $C$, defined as follows, vanishes
\begin{equation}\label{2.16}
C_{ijk} := A_{ij,k} - A_{ik,j}.
\end{equation} 
We recall that, for $n\geq 4$,  using the second Bianchi identity 
the Cotton tensor can also be defined as one of the possible divergences of the Weyl tensor:
\begin{equation}\label{2.17}
-\frac{n-2}{n-3}\sum\limits_{l} W_{lijk,l}= C_{ijk}.
\end{equation}

On any $n$-dimensional manifold $(M, g)$ $(n\geq 4)$,
in what follows a relevant role will be played by the Bach tensor, 
first introduced in general relativity by Bach \cite{bac} in early 1920s'.
 By definition, we have
\begin{equation}\label{2.18}
B_{ij}: = \frac{1}{n-3}W_{ikjl, kl} + \frac{1}{n-2}R_{kl}W_{ikjl}
\end{equation}
and by equation \eqref{2.17}, we have an equivalent expression of the Bach tensor:
\begin{equation}\label{2.19}
B_{ij} =\frac{1}{n-2} \left( C_{ijk, k}+R_{kl}W_{ikjl}\right).
\end{equation}

\subsection{Some basic facts for gradient Ricci solitons.}	
Now, let $(M^{n},g,f)$ be a gradient Ricci soliton and
equation \eqref{1.1} can be written as
\begin{equation}\label{2.20}
R_{ij}+f_{i,j}=\rho \delta_{ij}.
\end{equation}
We will recall some well-known facts of gradient Ricci solitons.

\begin{lemma} {\bf (Hamilton \cite{Ha95F})}\label{lemma2.1}
	Suppose that $(M^{n}, g, f)$ is a gradient Ricci soliton
	satisfying \eqref{2.20}, then the following formulae hold,
	\begin{equation}\label{2.21}
	\nabla R =2Ric(\nabla f, \cdot),
	\end{equation}
	\begin{equation}\label{2.22}
	R + \left| \nabla f \right|^ 2 - 2\rho f = C_0
	\end{equation}
	and
	\begin{equation}\label{2.23}
	\Delta R=\langle \nabla R, \nabla f \rangle + 2\rho R - 2|Ric|^{2},
	\end{equation}
	where $C_0$ is constant and $|Ric|^{2}=\sum\limits_{i,j}R_{ij}^{2}.$
	\end {lemma}
	
	The covariant 3-tensor $D$, introduced by H.-D. Cao and Q. Chen  in \cite{CC1}, 
	turns out to be a fundamental tool in the study of the geometry of gradient Ricci solitons 
	(more in general for gradient Einstein-type manifolds). In components it is defined as
	\begin{equation}\label{2.24}
	D_{ijk} = \frac{1} {n-2} (A_{ij} f_k- A_{ik} f_j) + \frac {1} {(n-1)(n-2)} (\delta_{ij}E_{kl} - \delta_{ik}E_{jl})f_l
	\end{equation}
	where $E_{ij}=R_{ij}-\frac{R}{2}\delta_{ij}$ is the Einstein tensor.
	This 3-tensor $D_{ijk}$ is closely tied to the Cotton tensor and played a significant role in \cite{CC1} and  \cite{CC2} on classifying locally conformally flat gradient steady solitons and Bach flat shrinking Ricci solitons.
	\begin{lemma}{\bf (Cao-Chen \cite{CC1,CC2})}\label{lemma2.2}
		Let $(M^{n}, g, f)$, $ n\geq 3$, be a complete gradient Ricci soliton
		satisfying \eqref{2.20}.  $D_{ijk}$ is closely related to the Cotton tensor and the Weyl tensor by
		\begin{equation}\label{2.26}
		D_{ijk}=C_{ijk}+f_l W_{lijk}.
		\end{equation}
		The  Bach tensor $B_{ij}$ can be expressed in terms of $D_{ijk}$ and the 
		Cotton tensor $C_{ijk}$
		\begin{equation} \label{2.25}
		B_{ij} =\frac{1}{n-2}\left( \sum_kD_{ijk,k}+\frac{n-3}{n-2} f_kC_{jik}\right).
		\end{equation}
	\end{lemma}
	
	\medskip
	Finally, by using equations \eqref{2.7}, \eqref{2.20}, \eqref{2.14}, \eqref{2.16} and \eqref{2.23}, 
	we immediately have the following lemma (e.g., see Kim \cite{Kim}).
		\begin{lemma}\label{lemma2.3}
		Let $(M^{n},g,f)$ be a gradient Ricci soliton with harmonic Weyl tensor, i.e.,
		$\delta W = \sum_l W_{lijk,l} =0$. Then the Cotten tensor $C$ vanishes and 
		the Schouten tensor $A$ is Codazzi.
		Moreover the following formula holds
		\begin{equation}
		\begin{aligned}\label{2.27}
		f_l R_{lijk}=&R_{ik,j}-R_{ij,k}\\
		=&\frac{1}{2(n - 1)} \left( R_j \delta_{ik} - R_k \delta_{ij}\right) \\ 
		=&\frac{1}{n - 1} f_l \left( R_{lj} \delta_{ik} - R_{lk} \delta_{ij}\right).
		\end{aligned}
		\end{equation}
	\end{lemma}
	
	\section{The basic local structure for $n$-dimensional gradient Ricci soliton with harmonic Weyl curvature}
	In this section, for any gradient Ricci soliton with harmonic Weyl tensor, 
	we first recall by the arguments of \cite{CC2,Kim} that 
	$\frac{\nabla f}{ | \nabla f | }$ is a Ricci-eigenvector field with its eigenvalue $\lambda_1$.
	There is a local function $s$ with $\nabla s =\frac{\nabla f}{ | \nabla f | }$, 
	such that $\lambda_1$ and $R$ are functions of $s$ only.
	Secondly, we show in Lemma \ref{lemma3.3} that
	the Ricci-eigenvalues $\lambda_i$, $i=1,2,\cdots, n$ locally depend only on the variable $s$.
	At last, we derive the integrability conditions when the local structure of the metric is a multiply warped product (see Theorem \ref{mulwar}).
	
	First of all, we have the next lemma (see also Lemma 2.5 in \cite{Kim}).
	
	\begin{lemma} \label{lemma3.1}
		In some neighborhood $U$ of each point in $\{ \nabla f \neq 0  \}$, 
		we choose an orthonormal frame field
		$\{E_1= \frac{\nabla f}{|\nabla f| }, E_2, \cdots,  E_n \}$ 
		with dual frame field $\{\omega_1= \frac{d f}{|\nabla f| }, \omega_2, \cdots, \omega_n\}$.
		The following properties hold.
		
		\smallskip
		{\rm (i)} $E_1= \frac{\nabla f }{|\nabla f | }$ is an eigenvector field of the Ricci tensor.
		
		\smallskip
		{\rm (ii)} The 1-form $\omega_1= \frac{d f}{|\nabla f| }$ is closed. So the distribution
		$V={\rm Span}\{E_2, \cdots , E_n\}$ 
		is integrable from the Frobenius theorem. 
		We denote by $L$ and $N$ the integrable curve of the vector field $E_1$ and the integrable submanifold of $V$ respectively.
		Then it holds locally $M=L\times N$ 
		and there exist local coordinates $(s, x_2, \cdots , x_n)$ of $M$ such that 
		$ ds= \frac{d f}{|\nabla f|}$,  $E_1 = \nabla s$,
		$V={\rm Span} \{\frac{\partial}{\partial x_2}, \cdots , \frac{\partial}{\partial x_n}\} $
		and
		$g=ds^2+\sum_{a}\omega_a^2$.
		
		\smallskip
		{\rm (iii)} $R$ and $R_{11}= Ric (E_1, E_1)$ 
		can be considered as functions of the variable $s$ only, 
		and we write the derivative in $s$ by a prime: 
		$f^{'} = \frac{df}{ds}$, etc..
	\end{lemma}
	
	\begin{proof}
		Noting that 
		$f_1=|\nabla f|\neq 0$, $f_a=0$ for $2\leq a\leq n$, \eqref{2.27} gives us
		\[
		0=f_1 R_{111a}= f_l R_{l11a}=\frac{1}{n - 1} f_l \left( R_{l1} \delta_{1a} - R_{la} \delta_{11}\right) =-\frac{1}{n - 1}f_1 R_{1a}.	
		\]
		Then $R_{1a}=0$, which implies $E_1= \frac{\nabla f }{|\nabla f | }$ is an eigenvector field of the Ricci cuvature. We proved {\rm (i)}.
		
		Making use of \eqref{2.21}, we get
		\begin{equation}\label{3.1}
		R_a=2R_{aj}f_j=2R_{a1}f_1+2\sum_{b\geq2}R_{ab}f_b=0,
		\end{equation}
		together with \eqref{2.22}, which shows 
		\[
		\left(  \left| \nabla f \right|^ 2\right)_a =-R_a + 2\rho f_a =0.
		\]
		Therefore 
		\[
		d\omega_1 = d\left( \frac{d f}{|\nabla f|}\right) 
		=-\frac{1}{2 |\nabla f|^{\frac{3}{2}}} d \left( |\nabla f|^{2}\right)  \wedge df
		=0
		\]
		and {\rm (ii)} is proved.
		
		Since $R_1=2R_{1j}f_j=2R_{11}f_1$ implies that
		\[
		R_{11}=\frac{1}{2 |\nabla f|}R_1,
		\]
		by combining with \eqref{3.1}, one can immediately get {\rm (iii)}.
	\end{proof}
	
	Let $(M^n,g,f)$ be a gradient Ricci soliton with harmonic Weyl curvature. 
	The condition harmonic Weyl curvature is equivalent to the Schouten tensor being a Codazzi tensor. 
	Derdzi\'{n}ski \cite{De} described the following: for a Codazzi tensor $A$ and a point $x$ in $M$, 
	let $E_A(x)$ be the number of distinct eigenvalues of $A_x$,
	and set 
	\[
	M_A = \{  x \in M \ | E_A {\rm \ is \ constant \ in \ a \ neighborhood \ of \ } x \}.
	\] 
	Then $M_A$ is an open dense subset of $M$;
	in each connected component of $M_A$, 
	the eigenvalues are well-defined and differentiable functions;
	eigenspaces of $A$ form mutually orthogonal differentiable distributions.
	
	On the other hand, as a gradient Ricci soliton, 
	$(M,g,f)$ is real analytic in harmonic coordinates; see \cite{Iv, HPW}. 
	Then if $f$ is not constant, $\{\nabla f \neq 0\}$ is open and dense in $M$.  
	
	Hence for each point $p\in M_A \cap \{ \nabla f \neq 0  \}$, there exists a neighborhood $U$ of $p$, 
	such that the number of distinct eigenvalues of the Ricci tensor is constant on $U$,
	and we assume that the multiplicities of $m$ distinct Ricci eigenvalues are $r_{1}, r_{2}, \cdots, r_{m}$, respectively, 
	except for the one, which is the eigenvalue with respect to the eigenvector $\nabla f$,
	where $1+r_{1}+r_{2}+ \cdots+ r_{m}=n $.
	Combining with Lemma \ref{lemma3.1},
	we can choose a locally orthonormal frame field
	$\{E_1= \frac{\nabla f}{|\nabla f| }, E_2, \cdots,  E_n\}$,
	with the dual $\{\omega_1= \frac{d f}{|\nabla f| }, \omega_2, \cdots, \omega_n\}$
	such that 
	\begin{equation}\label{3.2}
	R_{ij}=\lambda_i \delta_{ij}.
	\end{equation}
	Without loss of generality, we assume that 
	\[
	\lambda_{2}=\cdots=\lambda_{r_1+1}, \quad
	\lambda_{r_1+2}=\cdots=\lambda_{r_1+r_2+1},\quad
	\cdots  \quad 
	\lambda_{r_1+r_2+\cdots+r_{m-1}+2}=\cdots=\lambda_{n},
	\]
	and $\lambda_{2},\, \lambda_{r_1+2},\,\cdots,\,	\lambda_{r_1+r_2+\cdots+r_{m-1}+2}$
	are distinct.
	Next, our first goal is to prove that all Ricci-eigenvalues $ \lambda_i $ $(i=1, \cdots ,n)$
	depend only on the local variable $s$.
	\begin{lemma} \label{lemma3.2}
		Let $(M, g, f)$ be an $n$-dimensional gradient Ricci soliton  with harmonic Weyl curvature.
		For the above local frame field $\{ E_i \}$ in $M_{A} \cap \{ \nabla f \neq 0 \}$, we have that
	\begin{equation}\label{3.3}
	f''+\lambda_1=\rho,
	\end{equation}
	\begin{equation}\label{3.4}
	\omega_{1a}= \xi_a \omega_{a}
	\end{equation}
	and 
	\begin{equation}\label{3.5}
	R_{aa,1}
	=\left(\lambda_1-\lambda_a \right)\xi_a+\frac{R_1}{2(n-1)},
	\end{equation}
	where 
	\begin{equation} \label{3.6}
	\xi_a:= \frac{1}{ |\nabla f|} (\rho-\lambda_{a}).
	\end{equation} 
	\end{lemma}
	\begin{proof}

	From \eqref{2.5} and \eqref{2.6}, it follows that for $2\leq a \leq n$, $f_{1}=|\nabla f|=f'$, $f_{a}=0$, 
    \[
	f_{1,1}=f'' \quad  {\rm and} \quad  f'\omega_{1a}=f_{a,j}\omega_{j}.
    \]
	Putting $f_{i,j}=\left(\rho-\lambda_i \right)\delta_{ij}$ into the above equations, 
	one can immediately get \eqref{3.3} and \eqref{3.4}.
	Covariant derivatives of $R_{ij}$ can be shown by \eqref{2.9}
  \[
	R_{11,1}=\lambda_1'  \quad {\rm and} \quad 
   R_{1a,a}=\left(\lambda_1-\lambda_a \right)\xi_a.
  \]
	Combining with the harmonic condition $R_{aa,1}=R_{1a,a}+\frac{R_1}{2(n-1)}$ yields \eqref{3.5},
	and we have completed the proof of this lemma.
	\end{proof}

	Based on the above lemmas, we present the key lemma which was proved by Kim \cite{Kim} for the four dimensional case. 
	Here we include a proof 
	similar to the Shin's (see also Lemma 3 in \cite{S}).
	\begin{lemma} \label{lemma3.3}
		Let $(M, g, f)$ be an $n$-dimensional gradient Ricci soliton  with harmonic Weyl curvature.
		For the above local frame field $\{ E_i \}$ in $M_{A} \cap \{ \nabla f \neq 0 \}$,
		the Ricci eigenvalues $ \lambda_i $ $(i=1, \cdots ,n)$
		depend only on the local variable $s$, so do the functions $\xi_a$ for $2 \leq a\leq n$.
	\end{lemma}
	
	\begin{proof}
		Firstly, it follows from Lemma \ref{lemma3.1} that $R={\rm tr}(Ric)$ and $\lambda_1= R_{11}$ depend on $s$ only,
		and then we consider the Laplacian of the scalar curvature.	Calculating the covariant derivatives $R$ gives us
		$R_{1,1}=R''$ and $R_{a,b}=R' \xi_a \delta_{ab}$. Thus
	\[
	\Delta R=R_{1,1}+\sum\limits_{a}R_{a,a}=R''+R'\frac{(n-1)\rho-R+\lambda_1}{f'}
	\]
		depends on $s$ only. Then from \eqref{2.23},
		${\rm tr}(Ric^2)=|Ric|^{2}$ also depends on $s$ only.
		Therefore by this motivation, we denote
		$(Ric^k)_{ij}= R_{i i_1} R_{i_1 i_2} \cdots R_{i_{k-1} j}$
		with its trace ${\rm tr}(Rc^k)= \sum_{i=1}^n (\lambda_i)^k$,
       and then	our goal is to show that 
		the functions ${\rm tr}(Ric^k)$ for $k=1,2, \cdots, n-1$ depend on $s$ only, 
		which implies that $ \lambda_i $ for $i=1, \cdots ,n$ also depend only on $s$.
		
		Next, we apply the mathematical induction to prove the desired results.
		Assume ${\rm tr} (Ric^l)$ depends on $s$ only for $1\leq l \leq k$, 
		and then ${\rm tr}(Ric^{k+1})$ will be done.  In fact,
		\begin{align*}
		\left( \sum_{i=1}^n(R_{ii}^k)\right) _1=&\sum_{i=1}^nk R_{ii}^{k-1}(R_{ii})_1\\
		=&k\left(  R_{11}^{k-1}R_{11,1} + \sum_{a=2}^n R_{aa}^{k-1}R_{aa,1}\right).
		\end{align*}
		Putting \eqref{3.2}, \eqref{3.5} and \eqref{3.6} into the above equation,  we have
			\begin{align*}
		\left( \sum_{i=1}^n(R_{ii}^k)\right) _1
		=&k\lambda^{k-1}_1\lambda_1'
		+k\frac{1}{f'}\left[(\sum_{i}\lambda^{k+1}_i -\lambda^{k+1}_1)-(\lambda_1+\rho)(\sum_{i}\lambda^{k}_i -\lambda^{k}_1)\right]\\
		&+k\left( \rho\frac{1}{f'}\lambda_1+\frac{R'}{2(n-1)}\right)\left(\sum_{i}\lambda^{k-1}_i -\lambda^{k-1}_1\right),
		\end{align*}
		By assumption, every term except for $\sum_{i}\lambda_{i}^{k+1}$ in the above equation depends only on $s$.
		Thus ${\rm tr}(Ric^{k+1})=\sum_{i=1}^n \lambda_{i}^{k+1} $ is also a function of $s$ only, 
		and we have completed the proof of this lemma.
	\end{proof}
	
	We proceed to obtain the local structure of the metric for $n$-dimensional gradient Ricci solitons with harmonic Weyl curvature.
	First, we denote $[a]=\{b| \lambda_{b}=\lambda_{a}~ {\rm and } ~ b \neq1\}$ for $2\leq a\leq n$
	and make the following conventions on the range of indices:
	\[
	2 \leq a, b, c \cdots\leq n  \quad {\rm and} \quad  2 \leq \alpha, \beta,  \cdots\leq n
	\]
	where $[a]=[b]$, $[\alpha]=[\beta]$ and $[a]\neq[\alpha]$.
		\begin{lemma} \label{lemma3.3-}
		Let $(M, g, f)$ be an $n$-dimensional gradient Ricci soliton  with harmonic Weyl curvature.
		For the above local frame field $\{ E_i \}$ in $M_{A} \cap \{ \nabla f \neq 0 \}$,  for $p \notin \{1,[a],[\alpha]\}$.
		we have that  
		
		\begin{equation}\label{new}
			\left(\lambda_a-\lambda_\alpha \right)\omega_{a\alpha}(E_{p})=	\left(\lambda_a-\lambda_p \right)\omega_{ap}(E_{\alpha}),
		\end{equation}
		\begin{equation}\label{3.8}
			R_{1a1b}=-\left( \xi'_a+\xi^2_a\right) \delta_{ab},
		\end{equation}
		\begin{equation}\label{3.11}
			\xi'_a+\xi^2_a=-\frac{R'}{2(n-1)f'},
		\end{equation}
		\begin{equation}\label{3.13}
			\lambda_{1}=-f''+\rho=-(n-1)\left( \xi'_a+\xi^2_a\right)
		\end{equation}
		and
		\begin{equation}\label{3.12}
			\lambda'_a-\left(\lambda_1-\lambda_a \right)\xi_a=\frac{R'}{2(n-1)}.
		\end{equation}
		
	\end{lemma}
	
	\begin{proof}	
		We continue to compute the covariant derivatives of $R_{ij}$.
		For $a\neq b$
		\[
		R_{aa,1}=\lambda'_a,  \quad R_{aa,b}=R_{aa,\alpha}=0 \quad  and \quad R_{ab,i}=0,
		\]
		while for the different range of indices, 
		\[
		\left(\lambda_a-\lambda_\alpha \right)\omega_{a\alpha}
		= \sum_{k=1}^{n}R_{a\alpha,k}\omega_k 
		= R_{a\alpha,1}\omega_1+R_{a\alpha,a}\omega_a+R_{a\alpha,\alpha}\omega_\alpha
		+\sum_{i\neq 1,a,\alpha}R_{a\alpha,i}\omega_i.
		\]
		Since $R_{a\alpha,1}=0$ and $R_{a\alpha,a}=0$,
		$\omega_{a\alpha}(E_1)= \omega_{a\alpha}(E_a)=\omega_{a\alpha}(E_\alpha)=0$.
		Therefore, 
		\[
		\left(\lambda_a-\lambda_\alpha \right)\omega_{a\alpha}
		=\sum_{p\notin \{1,[a],[\alpha]\}}R_{a\alpha,p}\omega_p.
		\]
		Therefore, when $p\in \{1,[a],[\alpha]\}$, $\omega_{a\alpha}(E_p)=0$; while $p \notin \{1,[a],[\alpha]\}$
		\[
		\left(\lambda_a-\lambda_\alpha \right)\omega_{a\alpha}(E_p)=R_{a\alpha,p}.
		\]
		By (2.27), $R_{a\alpha,p}=R_{ap,\alpha}=R_{p\alpha,a}$,  then
		\[
		\left(\lambda_a-\lambda_\alpha \right)\omega_{a\alpha}(E_p)=\left(\lambda_a-\lambda_p\right)\omega_{ap}(E_\alpha)
		=\left(\lambda_\alpha-\lambda_p \right)\omega_{\alpha p}(E_a),
		\]
		which shows  \eqref{new} . 
		
		Next we will compute the curvature. From  \eqref{2.2} and \eqref{3.4},
		\[
		-\frac{1}{2}R_{1aij}\omega_{i}\wedge\omega_{j}=\left(\xi'_a+\xi^2_a \right)\omega_{1}\wedge\omega_{a}
		+ \sum_{\alpha\geq 2}\left(\xi_a-\xi_\alpha \right)\omega_{a\alpha}\wedge\omega_{\alpha},
		\]
		which shows \eqref{3.8} since $\omega_{a\alpha}(E_1)=0$.
		Making use of \eqref{2.27} again, then we see
		\[ 
		R_{1a1a}=\frac{R'}{2(n-1)f'} 
		\]
		which shows that all $R_{1a1a}$ are equal for $a=2, \cdots, n$, and \eqref{3.11} holds by \eqref{3.8}.
		Then in combination with \eqref{3.3}, we immediately
		get \eqref{3.13}. Putting $R_{aa,1}=\lambda'_a$ into \eqref{3.5} implies harmonic condition \eqref{3.12}.
		We have completed the proof of this lemma.
	\end{proof}
	
	\begin{remark} \label{kim}
		We would like to express our gratitude to Professor Kim for pointing out an error that when the eigenvalues $\lambda_2,\cdots,~\lambda_n$ are at least three distinct values, the equation $\omega_{a\alpha}=0$ is not always true. However, under the harmonic Weyl tensor condition, we can only establish the following exchange relation as shown in \eqref{new}.
		Under certain specific conditions, such as when at most two eigenvalues $\lambda_2,\cdots,~\lambda_n$ are distinct, $\omega_{a\alpha}=0$ holds true.
	\end{remark}

	At last, we are ready to derive the integrability conditions and the local structure of the metric as a multiply warped product. 
	For the multiplicities of distinct Ricci eigenvalues, without loss of generality, 
	we assume that $r_{1}=r_{2}=\cdots=r_{l}=1$ and 
	$r_{l+1},~ r_{l+2},~\cdots,~ r_{m} \geq2$.  
	\begin{theorem} \label{mulwar}
		Let $(M^n, g, f)$, $(n\geq 4)$, be an $n$-dimensional gradient Ricci soliton with harmonic Weyl curvature. 
		For each point $p \in M_A \cap \{ \nabla f \neq 0  \}$ and the above local frame field $\{ E_i \}$ in $M_{A} \cap \{ \nabla f \neq 0 \}$, $\omega_{a\alpha}=0$ for $k \notin \{1,[a],[\alpha]\}$ if and only if there exists a neighborhood $U$ of $p$ 
		such that 
		\[
		U = L\times\,  _{h_1}L_1 \cdots \times\, _{h_l}L_l\cdots \times\, _{h_{l+1}}N_{l+1}\times\cdots \times\, _{h_m}N_{m}
		\]
		is a multiply warped product furnished with the metric 
		\begin{equation} \label{3.10}
			g= ds^2 + h^2_1(s)  dt_1^2 + \cdots 
			+ h^2_l(s) dt_l^2+ h^2_{l+1}(s) \tilde{g}_{l+1}+ \cdots 
			+h^2_{m}(s) \tilde{g}_{m},
		\end{equation}
		where $h_j(s)$ are smooth positive functions for $1\leq j \leq m$,
		$dim ~L_\nu$=1 for $1\leq \nu \leq l$, and 
		$(N_\mu, \tilde{g}_{\mu})$ is an $r_\mu$-dimensional Einstein manifold 
		with the Einstein constant $(r_\mu-1)k_\mu$ for $l+1\leq \mu \leq m$.
		Let $\lambda_i$ $(i=1,\cdots, n)$ be the Ricci-eigenvalues, $\lambda_1$ being the Ricci-eigenvalue with respect to the gradient vector $\nabla f$,
		and functions $\xi_a$ be given by \eqref{3.6}.	
		Then the following integrability conditions hold
		\begin{equation}
			\xi'_a+\xi^2_a=-\frac{R'}{2(n-1)f'},
		\end{equation}
		\begin{equation}
			\lambda'_a-\left(\lambda_1-\lambda_a \right)\xi_a=\frac{R'}{2(n-1)},
		\end{equation}
		\begin{equation}
			\lambda_{1}=-f''+\rho=-(n-1)\left( \xi'_a+\xi^2_a\right)
		\end{equation}
		and
		\begin{equation}\label{3.14}
			\lambda_a=-f'\xi_a+\rho=- \xi'_a-\xi_a \sum^{n}_{i=2} \xi_i +( r-1)\frac{k}{h^2}.
		\end{equation}
		Here $2\leq a\leq n$ and $\xi_a=h'/h$;  in \eqref{3.14}, when $a=2, \dots, l+1$,  then 
		$r=1$ and $h=h_{a-1}$;
		when $a\in \left[ l+r_{l+1}+\cdots+r_{\mu-1} +2\right]$ ,  then 
		$r=r_\mu$, $h=h_{\mu} $ and $k=k_{\mu} $.
	\end{theorem}
	
	\begin{proof}
		Adequacy is obvious; we only need to prove necessity.
		From the structure equation 
		\[
		-\frac{1}{2}R_{a\alpha ij}\omega_{i}\wedge\omega_{j}=\xi_a\xi_\alpha \omega_{a}\wedge\omega_{\alpha},
		\]
		which implies	
		\begin{equation}\label{3.9}
			R_{a\alpha b \beta}=-\xi_a\xi_\alpha \delta_{ab}\delta_{\alpha\beta}.
		\end{equation}

		Moreover, by equations \eqref{2.1}, \eqref{3.4} and \eqref{new}, 
		\[
		d\omega_\alpha=\sum_\beta\left( \xi_\alpha\delta_{\alpha\beta}
		+\omega_{\alpha\beta} \right) \wedge \omega_{\beta}  \equiv 0 
		~ ( {\rm mod} ~ \omega_\beta).
		\]
		By $d\omega_1=0$ and the Frobenius theorem, the distribution
		$V_a={\rm Span} \{E_b: b\in [a] \} $ is integrable.
		It is worth noting that $\omega_{a\alpha}=0$ yields more information than that Ricci-eigenspaces form mutually orthogonal differentiable distributions. Denoting $N_a$ to be the integrable submanifold of $V_a$, 
		we consider the $r$-dimensional isometric immersion submanifold $N_a$ with the metric 
		$\bar{g}=\sum_{b \in [a]} \omega^2_{b}$
		of the manifold $(M,g)$.
		
		When the multiplicity $r$ of the Ricci-eigenvalue $\lambda_{a}$ is not less than two, it follows from \eqref{3.8} and \eqref{3.9} that 
		\begin{equation}
			\begin{aligned}\label{3.15}
				R_{ab}=&R_{a1 b1}+\sum_{\alpha}R_{a\alpha b\alpha}+\sum_{c} R_{acbc}\\
				=&\left[ -\left( \xi'_a+\xi^2_a \right) -\xi_a \sum_{\alpha} \xi_\alpha \right] \delta_{ab} 
				+\sum_{c} R_{acbc}.
			\end{aligned}
		\end{equation}
		Making use of equations \eqref{3.4} and \eqref{new} again, Gauss equations imply
		\begin{equation}\label{3.16}
			\bar{R}_{abcd}={R}_{abcd}+
			\xi^2_{a}\left( \delta_{ac}\delta_{bd}-\delta_{ad}\delta_{bc} \right),
		\end{equation}
		where $\bar{R}_{abcd}$ is the curvature tensor of $\left( N_a, \bar{g}\right) $.
		Actually, equation \eqref{new} means that the submanifold $(N_a,\bar{g})$ of the Riemannian manifold 
		$(N,~\sum_{i\neq1} \omega^2_{i})$ is totally geodesic, where $N$ is the integrable submanifold generated by the distribution
		$V={\rm Span} \{E_2, \cdots, E_n \} $, see Lemma \ref{lemma3.1}.
		Taking trace in \eqref{3.16} and then plugging \eqref{3.15} into it, one can get
		\begin{equation*}
			\begin{aligned}
				\bar{R}_{ac}=&\sum_{b}R_{abcb}+\left( r-1\right) \xi^2_{a}\delta_{ac}\\
				=& \left[\lambda_a+\left( \xi'_a+\xi^2_a \right) +\xi_a \sum_{\alpha \notin [a]} \xi_\alpha +( r-1) \xi^2_{a} \right] \delta_{ac}
			\end{aligned}
		\end{equation*}
		which depends on $s$ only. Hence, the metric $\bar{g}$ is Einstein.
		Then we can assume that $\bar{g}=h^2(s)\tilde{g}$, where $h(s)$ is a positive function and $\tilde{g}$ is Einstein with the Einstein constant $(r-1)k$. So 
		\[
		\bar{R}_{ac}=\frac{(r-1)k}{h^2}\delta_{ac}
		\]
		and 
		\begin{equation}\label{3.17}
			\lambda_a=- \xi'_a-\xi_a \sum^{n}_{i=2} \xi_i +( r-1)\frac{k}{h^2}.
		\end{equation}
		Also, it follows from the structure equation that $\xi_a=h'/h$.
		
		When the Ricci tensor eigenvalue $\lambda_a$ has single multiplicity,
		by using equations \eqref{2.1}, \eqref{3.4} and \eqref{new}, 
		we calculate the exterior differential of 1-form $\omega_a$: 
		\[
		d\omega_a=\omega_{i}\wedge\omega_{ia}=\xi_a\omega_{1}\wedge\omega_a.
		\]
		Setting the positive function $h(s)$ such that $\xi_a=h'/h$, and we see that
		1-form $\frac{1}{h}\omega_a$ is closed. Then there is a local function $t$ satisfying
		$dt =\frac{1}{h(s)} \omega_a$ and 
		\begin{equation}\label{3.18}
			\lambda_a=R_{1a1 a}+\sum_{\alpha \notin [a]}R_{a\alpha a\alpha}
			=- \xi'_a-\xi_a \sum^{n}_{i=2} \xi_i.
		\end{equation}
		
		Consequently, we obtain that if $\omega_{a\alpha}=0$, the metric is a multiply warped product as seen in \eqref{3.10}. 
		Furthermore, it follows that \eqref{3.14} holds from equations \eqref{2.20}, \eqref{3.17} and \eqref{3.18}. 
		In fact, by comparing with \eqref{3.18}, there is no term $( r-1)\frac{k}{h^2}$ in \eqref{3.17}. 
		However, in the case $r=1$, this term vanishes for any constant $k$, and then the Ricci-eigenvalue $\lambda_{a}$ also satisfies the form like \eqref{3.17}.
		Hence, for convenience sometimes we add this term. We have completed the proof of this theorem.
	\end{proof}

	\section{The local structure  of the case with more than two distinct Ricci-eigenvalues}
	In this section, and Sections 5-6, we will give the local classification of $n$-dimensional gradient Ricci solitons with harmonic Weyl curvature and multiply warped product metric of eigenspaces with the Ricci tensor, according to numbers and multiplicities of distinct Ricci-eigenvalues written as $\lambda_a$, $a=2,\cdots, n$.
	To avoid repetition, unless stated otherwise, the Ricci-eigenvalues mentioned in the following discussion do not include $\lambda_1$, which is the eigenvalue with respect to the gradient vector of the potential function.

	In this section, we shall study the case when $\lambda_2,~\lambda_3,~\cdots,~\lambda_n$ are at least three mutually different, but it turns out that this case cannot occur.
    
			\smallskip
	First, we recall that as a gradient Ricci soliton, $(M,g,f)$ is real analytic in harmonic coordinates \cite{Iv}, i.e., $g$ and $f$ are real analytic (in harmonic coordinates).
	To exploit the real analyticity, we shall use the following simple facts:
	
	{\rm (i).} 
	If an analytic function $P$ is not constant, $\{ \nabla P \neq 0  \}$ is open and dense in $M$. 
	
	{\rm (ii).} 
	If $P \cdot Q$ equals zero (identically) on an open connected set $U$ for two real analytic functions $P$ and $Q$, 
	then either $P$ equals zero on $U$ or $Q$ equals zero on $U$.
	
\medskip	
Proceeding, we analyze the integrability conditions in Theorem \ref{mulwar}.
Assume that $\lambda_a$ and $\lambda_\alpha$ are mutually different of multiplicities $r_1$ and $r_2$, i.e.,
\[
\lambda_{a}:=\lambda_{2}= \cdots= \lambda_{r_1+1} \quad {\rm and} \quad
\lambda_{\alpha}:=\lambda_{r_1+2}= \cdots=\lambda_{r_1+r_2+1};
\]
Here we make the following conventions on the range of indices:
\[
2 \leq a, b,  \cdots\leq ({r_1+1}); \quad ({r_1+2}) \leq \alpha, \beta,  \cdots\leq ({r_1+r_2+1});\quad 
\] 
Denote $\xi_a:=X$ and $ \xi_\alpha:=Y$, from Section 3, and then they satisfy the following integrability conditions:
\begin{equation}\label{4.1}
X'+X^2=Y'+Y^2=\xi'_i+\xi^2_i,
\end{equation}
\begin{equation}\label{4.2}
\lambda_{1}=-f''+\rho=-(n-1)\left( X'+X^2 \right),
\end{equation}
\begin{equation}\label{4.3}
\lambda_{a}=-f'X+\rho=-( X'+X^2)+X^2+(r_1-1)\frac{k_1}{h^2_1} 
-X\sum^n_{i=2}\xi_i,
\end{equation}
\begin{equation}\label{4.4}
\lambda_{\alpha}=-f'Y+\rho=-(Y'+Y^2)+Y^2+(r_2-1)\frac{k_2}{h^2_2}
-Y\sum^n_{i=2}\xi_i
\end{equation}
 and 
\begin{equation}\label{4.5}
\lambda'_a-\left(\lambda_1-\lambda_a \right)X=\lambda'_\alpha-\left(\lambda_1-\lambda_\alpha \right)Y.
\end{equation}
By using the above basic facts, it is easy to get the following equations:

\begin{lemma}\label{lemma4.1}	
	Let $(M^n, g, f)$, $(n\geq 4)$, be an $n$-dimensional gradient Ricci soliton with harmonic Weyl curvature  and multiply warped product metric of eigenspaces with the Ricci tensor,
in some neighborhood $U$ of $p\in M_A \cap \{ \nabla f \neq 0  \}$, 
suppose $\lambda_a$ and $\lambda_\alpha$ are mutually different Ricci eigenvalues with multiplicities $r_1$ and $r_2$.
The following identities hold:
\begin{equation}\label{4.6}
(r_1-1)\frac{k_1}{h^2_1}-(r_2-1)\frac{k_2}{h^2_2}=(X-Y)\left[ \sum^n_{i=2}\xi_i-(X+Y)-f' \right],
\end{equation}
\begin{equation}\label{4.7}
 (r_1-1)\frac{k_1}{h^2_1}+ (r_2-1)\frac{k_2}{h^2_2}
 =2(X'+X^2 ) +\rho+\sum^n_{i=2}\xi^2_i-(X^2+Y^2),
\end{equation}
\begin{equation}\label{4.8}
\begin{aligned}
&(r_1-1)\frac{k_1}{h^2_1}X- (r_2-1)\frac{k_2}{h^2_2}Y\\
=&(X-Y)\left\{ (X'+X^2 ) +\rho+(X+Y)\left[ \sum^n_{i=2}\xi_i-(X+Y) -f'\right]+XY\right\},
\end{aligned}
\end{equation}
\begin{equation}\label{4.9}
-(r_1-1)\frac{k_1}{h^2_1}Y+(r_2-1)\frac{k_2}{h^2_2}X
=(X-Y)[(X'+X^2 )+\rho+XY],
\end{equation}
\begin{equation}\label{4.10}
\begin{aligned}
&(r_1-1)\frac{k_1}{h^2_1}X- (r_2-1)\frac{k_2}{h^2_2}Y\\
=&(X-Y)\left[(X'+X^2 )+\sum^n_{i=2}\xi^2_i-(X^2+Y^2)-XY\right],
\end{aligned}
\end{equation}
\begin{equation}\label{4.11}
\begin{aligned}
&\left[(r_1-1)\frac{k_1}{h^2_1}- (r_2-1)\frac{k_2}{h^2_2} \right](X+Y)\\
=&(X-Y)\left[\sum^n_{i=2}\xi^2_i-(X^2+Y^2)-2XY-\rho \right]
\end{aligned}
\end{equation}	
and 
 \begin{equation}\label{4.12}
 \sum^n_{i=2}\xi^2_i-\rho=(X+Y)\left( \sum^n_{i=2}\xi_i -f' \right).
 \end{equation}
\end{lemma}
\begin{proof}
 First, subtracting \eqref{4.3} from \eqref{4.4} gives us 
\begin{equation}
	\begin{aligned}\label{4.13}
	\lambda_{\alpha}-\lambda_a=&f'(X-Y)\\
	=&(X-Y)\left[ \sum^n_{i=2}\xi_i-(X+Y) \right]
	-\left[(r_1-1)\frac{k_1}{h^2_1}-(r_2-1)\frac{k_2}{h^2_2}\right],
\end{aligned}
\end{equation}
which implies \eqref{4.6}.
Noting that $h'_1/h_1=X$ and $h'_2/h_2=Y$
and differentiating \eqref{4.13} lead to
\begin{equation}\label{4.14}
\begin{aligned}
&( \lambda_{\alpha}-\lambda_a )'=\left[f'(X-Y)\right]'\\
=&(X-Y)\Bigg\{(n-3)(X'+X^2 )-(X+Y)\left[ \sum^n_{i=2}\xi_i-(X+Y) \right]\\
&-\left[ \sum^n_{i=2}\xi^2_i-(X^2+Y^2)\right] \Bigg\}
+2\left[ (r_1-1)\frac{k_1}{h^2_1}X-(r_2-1)\frac{k_2}{h^2_2}Y\right].
\end{aligned}
\end{equation}
On the other hand, applying \eqref{4.1}, \eqref{4.2} and \eqref{4.13}, we see
\begin{equation}\label{4.15}
\begin{aligned}
&\left[f'(X-Y)\right]'=f''(X-Y)-f'(X-Y)(X+Y)\\
=&(X-Y)\left\{[(n-1)(X'+X^2 )+\rho]+(X+Y)\left[ \sum^n_{i=2}\xi_i-(X+Y) \right]   \right\}\\
&+(X+Y))\left[ (r_1-1)\frac{k_1}{h^2_1}-(r_2-1)\frac{k_2}{h^2_2}\right].
\end{aligned}
\end{equation}
Comparing with equations \eqref{4.14} and \eqref{4.15} gives us \eqref{4.7},
where the fact $X\neq Y$ was used. 

\smallskip
Next, we consider the harmonic condition. It follows from \eqref{4.2}, \eqref{4.3} and \eqref{4.4} that 
\begin{equation}
\begin{aligned}\label{4.16}
&(\lambda_1-\lambda_\alpha)Y-(\lambda_1-\lambda_a)X\\
=&(X-Y)\left\{(n-2)(X'+X^2 )-(X+Y)\left[ \sum^n_{i=2}\xi_i-(X+Y) \right]-XY\right\}\\
&+\left[ (r_1-1)\frac{2k_1}{h^2_1}X-(r_2-1)\frac{2k_2}{h^2_2}Y\right].
\end{aligned}
\end{equation}
Since $\lambda'_\alpha-\lambda'_a=\left(\lambda_1-\lambda_\alpha \right)Y-\left(\lambda_1-\lambda_a \right)X$ from  harmonic condition \eqref{4.5},
by comparing \eqref{4.15} and \eqref{4.16} we obtain \eqref{4.8}.
Putting
\[
f'(X-Y)
=(X-Y)\left[ \sum^n_{i=2}\xi_i-(X+Y) \right]
-\left[(r_1-1)\frac{k_1}{h^2_1}-(r_2-1)\frac{k_2}{h^2_2}\right]
\]
into \eqref{4.8} implies \eqref{4.9}.

On the other hand, we can obtain \eqref{4.10} by combining \eqref{4.14} and \eqref{4.16}.
Meanwhile, subtracting \eqref{4.10} from \eqref{4.9} shows \eqref{4.11}, while \eqref{4.12} is proved by subtracting \eqref{4.8} from \eqref{4.10}, and we have completed the proof of this lemma.
\end{proof}	
	
Proceeding, we will apply equations \eqref{4.1}, \eqref{4.2} and \eqref{4.12} to obtain the following desired result.		
\begin{theorem}\label{thm4.1}
Let $(M^n, g, f)$, $(n\geq 4)$, be an $n$-dimensional gradient Ricci soliton with harmonic Weyl curvature  and multiply warped product metric of eigenspaces with the Ricci tensor.
 Then in some neighborhood $U$ of $p \in M_A \cap \{ \nabla f \neq 0  \}$, the Ricci eigenvalues $\lambda_2,~\lambda_3,~\cdots,~\lambda_n$ cannot be more than two distinct. 
\end{theorem}
	
\begin{proof}
If not, we assume that $\lambda_2,~\lambda_3,~\cdots,~\lambda_n$ are at least three mutually different, and denote by $\lambda_a$, $\lambda_\alpha$ and $\lambda_p$ with multiplicities $r_1$, $r_2$ and $r_3$.
For convenience, we also denote $\xi_a:=X$, $\xi_\alpha:=Y$ and  $\xi_{p}:=Z$. by \eqref{4.12}, with the assumption of Lemma \ref{lemma4.1}, we see that
\[
\sum^n_{i=2}\xi^2_i-\rho=(X+Y)\left( \sum^n_{i=2}\xi_i -f' \right)
=(X+Z)\left( \sum^n_{i=2}\xi_i -f' \right)
=(Y+Z)\left( \sum^n_{i=2}\xi_i -f' \right).
\]
Therefore, it follows that
\begin{equation}\label{4.17}
\sum^n_{i=2}\xi_i =f' 
\end{equation}
since $X$, $Y$ and $Z$ are distinct, and then we have
\begin{equation}\label{4.18}
\sum^n_{i=2}\xi^2_i =\rho.
\end{equation}
 On the other hand, by \eqref{4.1}, \eqref{4.2} and differentiating \eqref{4.17}, we obtain 
\[
-\sum^n_{i=2}\xi^2_i =\rho,
\]
hence, which yields $\sum^n_{i=2}\xi^2_i=0$ by combining with \eqref{4.18}.
This is a contradiction since $X,~Y,~Z$ are distinct.
We have completed the proof of this theorem. 
\end{proof}

\section{The local structure of the case with two distinct Ricci-eigenvalues}
	
	In this section we begin to study the case when there are exactly two distinct Ricci values in the eigenvalues 
  $\lambda_2, \cdots, \lambda_n$. Types {\rm (ii)} and {\rm (iii)} of Theorem \ref{local} come from this section.
	
	First of all, we need the following lemmas to prepare for the local structure of the case.
	
	\begin{lemma}\label{lemma5.1} 
		Let $(M^n, g, f)$, $(n\geq 4)$, be an $n$-dimensional gradient Ricci soliton with harmonic Weyl curvature.
		Assume there are exactly two distinct values in the Ricci eigenvalues $\lambda_2, \cdots, \lambda_n$, denoted by $\lambda_a$ and $\lambda_\alpha$ of multiplicities $r_1$ and $r_2:=n-r_1-1$ in some neighborhood $U$ of $p \in M_A \cap \{ \nabla f \neq 0  \}$.
		Then the functions $X$ and $Y$ satisfy the following equations:
		\begin{equation}\label{5.1}
	(n-1)XY+\rho=f'(X+Y),
		\end{equation}
	\begin{equation}\label{5.2}
		X'+X^2+XY=0
		\end{equation}		
{\rm and}		
\begin{equation}\label{5.3}
		XY[(r_1-1)X^2+(r_2-1)Y^2-2XY-\rho]=0.
\end{equation}	
	\end{lemma}
	
	\begin{proof}
In this case, equations \eqref{4.6}-\eqref{4.12} in Lemma \ref{lemma4.1} become the following basic identities:
\begin{equation}\label{5.4}
(r_1-1)\frac{k_1}{h^2_1}-(r_2-1)\frac{k_2}{h^2_2}=(X-Y)[ (r_1-1)X+(r_2-1)Y-f' ],
\end{equation}
\begin{equation}\label{5.5}
(r_1-1)\frac{k_1}{h^2_1}+ (r_2-1)\frac{k_2}{h^2_2}
=2(X'+X^2 ) +\rho+[ (r_1-1)X^2+(r_2-1)Y^2 ],
\end{equation}
\begin{equation}\label{5.6}
\begin{aligned}
&(r_1-1)\frac{k_1}{h^2_1}X- (r_2-1)\frac{k_2}{h^2_2}Y\\
=&(X-Y)\{ (X'+X^2 ) +\rho+(X+Y) [ (r_1-1)X+(r_2-1)Y-f']+XY\},
\end{aligned}
\end{equation}
\begin{equation}\label{5.7}
-(r_1-1)\frac{k_1}{h^2_1}Y+(r_2-1)\frac{k_2}{h^2_2}X
=(X-Y)[(X'+X^2 )+\rho+XY],
\end{equation}
\begin{equation}\label{5.8}
\begin{aligned}
&(r_1-1)\frac{k_1}{h^2_1}X- (r_2-1)\frac{k_2}{h^2_2}Y\\
=&(X-Y)\left[(X'+X^2 )+(r_1-1)X^2+(r_2-1)Y^2-XY\right],
\end{aligned}
\end{equation}
\begin{equation}\label{5.9}
\begin{aligned}
&\left[(r_1-1)\frac{k_1}{h^2_1}- (r_2-1)\frac{k_2}{h^2_2} \right](X+Y)\\
=&(X-Y)\left[(r_1-1)X^2+(r_2-1)Y^2-2XY-\rho \right]
\end{aligned}
\end{equation}	
and 
\begin{equation}\label{5.10}
r_1X^2+r_2Y^2-\rho=(X+Y)\left( r_1X+r_2Y-f' \right).
\end{equation}
We can immediately simplify \eqref{5.10} to get \eqref{5.1}.
By \eqref{4.1}, \eqref{4.2} and differentiating \eqref{5.1}, we obtain \eqref{5.2}.

	Next, by \eqref{4.1}, \eqref{5.2} and differentiating \eqref{5.5}, we see	
	\begin{equation}\label{5.11}
	\begin{aligned}
	&(r_1-1)\frac{k_1}{h^2_1}X+(r_2-1)\frac{k_2}{h^2_2}Y\\
	=&-2XY(X+Y)+[(r_1-1)X^2(X+Y)+(r_2-1)Y^2(X+Y)]\\
	=&(X+Y)[(r_1-1)X^2+(r_2-1)Y^2-2XY].
	\end{aligned}
	\end{equation}	
	Equations \eqref{5.9} and \eqref{5.5} yield that
	\[
	(r_1-1)\frac{k_1}{h^2_1}(X+Y)=[(r_1-1)X^2+(r_2-1)Y^2-2XY]X+\rho Y
	\]
	and
	\[
	(r_2-1)\frac{k_2}{h^2_2}(X+Y)=[(r_1-1)X^2+(r_2-1)Y^2-2XY]Y+\rho X.
	\]
 Multiplying both sides of \eqref{5.11} by $(X+Y)$ and then putting the above two equation into it, we obtain that
	\[
	2XY[(r_1-1)X^2+(r_2-1)Y^2-2XY-\rho]=0.
	\]
	Therefore, the proof of this lemma is completed.
\end{proof}

\begin{lemma}\label{lemma5.2} 
	Let $(M^n, g, f)$, $(n\geq 4)$, be an $n$-dimensional gradient Ricci soliton with harmonic Weyl curvature which satisfies the hypothesis of Lemma \ref{lemma5.1}.
	Then functions $X$ and $Y$ also satisfy the following inequalities
	\begin{equation}\label{5.12}
     X+Y\neq0
	\end{equation}
and
	\begin{equation}\label{5.13}
	(n-1)XY+\rho \neq 0.
	\end{equation}
\end{lemma}

\begin{proof}
	First, if $X+Y=0$, by combining with $X'+X^2=Y'+Y^2$, we have $X'=Y'=0$.
	$X$ and $Y$ are nonzero constant since they are distinct.
	Plugging them into \eqref{5.1}, then
	\[
	(n-1)X^2-\rho=0.
	\]
	From \eqref{4.2}, we see $f''=2\rho$. It follows from \eqref{4.3} that
	\[
	-f'X+\rho=-(r_1-r_2)X^2+(r_1-1) \frac{k_1}{h^2_1}
	\]
	and then differentiating it, one can easily get $f''=2(r_1-1) \frac{k_1}{h^2_1}$. Therefore 
	\[
	(r_1-1) \frac{k_1}{h^2_1}=\rho.
	\]
	Similarly, it holds that $(r_2-1) \frac{k_2}{h^2_2}=\rho$. 
	Putting them into \eqref{5.4}, we see
	\[
	f'=(r_1-r_2)X
	\]
	is constant, which implies that $f''=0$.
	Hence
	\[
	(r_1+r_2)X^2=\rho=0
	\]
	and then $X=0$, which is a contradiction. Thus \eqref{5.12} is proved, and 
	\eqref{5.13} immediately follows from \eqref{5.1}. The proof of this lemma is completed. 	
\end{proof}

\begin{lemma}\label{lemma5.3} 	
	Let $(M^n, g, f)$, $(n\geq 4)$, be an $n$-dimensional gradient Ricci soliton with harmonic Weyl curvature
	which satisfies the hypothesis of Lemma \ref{lemma5.1}.  Then we have the following properties:
	{\item {\rm (i)}} The equality
$(r_1-1)\frac{k_1}{h^2_1}= (r_2-1)\frac{k_2}{h^2_2}$
holds if and only if 
	\begin{equation}\label{5.14}
	(r_1-1)X^2+(r_2-1)Y^2-2XY-\rho = 0.
	\end{equation}

	{\item {\rm (ii)}}
	If	$(r_1-1)\frac{k_1}{h^2_1}= (r_2-1)\frac{k_2}{h^2_2}$,
	then the Ricci soliton is steady {\rm (}i.e., $\rho=0${\rm )} and 
	\begin{equation}\label{5.15}
	(r_1-1)X^2+(r_2-1)Y^2-2XY= 0.
	\end{equation}
\end{lemma}

\begin{proof}
	First, noting \eqref{5.12}, (i) immediately holds by \eqref{5.9}.
	
	Next, if $(r_1-1)\frac{k_1}{h^2_1}=(r_2-1)\frac{k_2}{h^2_2}$, from \eqref{5.4}, we have
	\[
	(r_1-1)X+(r_2-1)Y=f'.
	\]
	Then differentiating this equation gives us
	\[
	(r_1-1)X^2+(r_2-1)Y^2-2XY+\rho=0,
	\]
	where \eqref{4.1} and \eqref{4.2} were used.
	By combining with equation \eqref{5.14}, we immediately have (ii) holds. 
The proof of this lemma is completed. 	
\end{proof}
Proceeding we will discuss two cases according to 
whether one of the two Ricci-eigenvalues is of single multiplicity.

\subsection{The multiplicities of two Ricci-eigenvalues $\lambda_a$ and $\lambda_\alpha$ are more than one}
In this subsection, we will study the case when the multiplicities of two Ricci eigenvalues 
$\lambda_a$ and $\lambda_\alpha$ are more than one in some neighborhood $U$ of a point $p$ in $M_A \cap \{ \nabla f \neq 0  \}$.
We have the following local classification result, 
which forms type {\rm (ii)} of Theorem \ref{local} for $3\leq r\leq n-2$.
	
\begin{theorem} \label{twodis}
	Let $(M^n, g, f)$, $(n\geq 4)$, be an $n$-dimensional gradient Ricci soliton with harmonic Weyl curvature.
	Assume in some neighborhood $U$ of $p$ in $M_A \cap \{ \nabla f \neq 0  \}$, there are exactly two distinct values in the Ricci eigenvalues $\lambda_2, \cdots, \lambda_n$, denoted by $\lambda_a$ and $\lambda_\alpha$, where both of the multiplicities are bigger than one. 
	Then $(U,g,f)$ can not be steady (i.e., $\rho\neq0$). The two distinct eigenvalues are exactly  0 and $\rho$ with the multiplicities $(r+1)$ and $(n-r-1)$, respectively, where $2\leq r\leq n-3$. 
	Also, $\nabla f$ is a null Ricci-eigenvector.
	
	Moreover, $(U,g)$ is locally isometric to a domain in 
	$\mathbb{R}^{r+1}\times N^{n-r-1}$
	with $g= ds^2 + s^2d\mathbb{S}^2_{r}+\tilde{g} $,
	where $\left(N^{n-r-1}, \tilde{g}\right)$ 
	is an $(n-r-1)$-dimensional Einstein manifold
	with the Einstein constant $\rho\neq 0$.
	The potential function is given by $f = \frac{\rho}{2}s^2$ modulo a constant.
	\end{theorem}

	\begin{proof}
	In this case, we denoted the two distinct Ricci eigenvalues by $\lambda_a$ and $\lambda_\alpha$ 
	with multiplicity $r_1=r$ and $r_2=(n-r-1)$, where $r_i\geq2,\ i=1,\ 2$. 
	Then it follows from Theorem \ref{mulwar} that 
	$(U,g)$ is isometric to a domain in 
	$\mathbb{R}\times N^{r}_1 \times N^{n-r-1}_2$
	with
	\[
     g= ds^2 + h^2_1(s)\tilde{g}_{1}+h^2_{2}(s) \tilde{g}_{2}, 
	\]
	where $(N_i, \tilde{g}_{i})$ is an $r_i$-dimensional Einstein manifold 
	with the Einstein constant $(r_i-1)k_i$ for $i=1,\ 2$.
	
Based on what we have discussed,  we first claim that one of the functions $X$ and $Y$ vanishes,
where $X:=\xi_a$ and $Y:=\xi_\alpha$ are the functions corresponding to the Ricci eigenvalues 
$\lambda_a$ and $\lambda_\alpha$, respectively.

In fact, if not, then neither X nor Y is zero. It follows from \eqref{5.3} that
		\[
		(r_1-1)X^2+(r_2-1)Y^2-2XY-\rho=0.
		\]
		Therefore,
		\[
		\rho=(r_1-1)X^2+(r_2-1)Y^2-2XY=(r_1-2)X^2+(r_2-2)Y^2+(X-Y)^2
		\]
		is positive since $r_1,~r_2\geq2$.
		On the other hand, from Lemma \ref{lemma5.3}, we see $\rho=0$,
		which is a contradiction.
	
	 Next, without loss of generality, we assume that $X\neq0$ and $Y=0$.
	$Y=0$ implies that $h_2$ is constant and $X^{'} + X^2=0$. Thus $X= \frac{1}{s-c_1}$ and 
	$h_1= c_{h_1} (s-c_1)$ after with integrating $X= \frac{h^{'}_1}{h_1}$ for constants $c_1$ and $c_{h_1}$.
	By \eqref{4.2}, we have
	\[
	\lambda_1=-f''+\rho=0,
	\]
	which yields that $f''=\rho$ and
	\[
	f(s) = \frac{1}{2} \rho (s-c_1)^2+C_1.
	\]
	Putting $Y=0$ into \eqref{4.3} and \eqref{4.4} respectively shows
	\begin{equation}\label{5.16}
	\lambda_a=-f'X+\rho=-(r-1)(X^2-\frac{k_1}{h^2_1} )
	\end{equation}
	and
    \[
   \lambda_\alpha=\rho=(n-r-2)\frac{k_1}{h^2_2}.
    \]
	Noting that $-f'X+\rho=0$, by \eqref{5.16}, we see $\lambda_a=0$
	and $\lambda_\alpha=\rho\neq0$. 
	
	Therefore, the Ricci curvature components and the scalar curvature are as follows:
	$R_{11} = R_{aa} =0$,
	$R_{\alpha\alpha}=\rho$, $R_{ij} =0$ 
	$(i \neq j)$ and 
	$R= (n-r-1)\rho $.
	
	 Furthermore, since $r\geq2$, \eqref{5.16} shows $X^2-\frac{k_1}{h^2_1}=0$, and then $\frac{k_1}{c^2_{h_1}}=1$.
	Hence, $(U,g)$ is isometric to a domain in 
	$\mathbb{R}^{1}\times N^{r}_1 \times N^{n-r-1}_2$
	with $g= ds^2 + s^2\tilde{g}_1+\tilde{g}_2 $,
	where $\left(N^{r}_1, \tilde{g}_1\right)$ and $\left(N^{n-r-1}_2, \tilde{g}_2\right)$ 
	are Einstein manifolds with the Einstein constants $(r-1)$ and $\rho\neq 0$, respectively.
	Finally, we consider the manifold $\mathbb{R}^{1}\times N^{r}_1$
	with $\bar{g}=ds^2 + s^2\tilde{g}_1$. From \eqref{3.5}, we have 
	$f_{1,1}=f_{a,a}=\rho\neq0$ and $f_{a,b}=f_{1,a}=0$, i.e.,
	\[
	Hess_{\bar{g}}(f)=\rho\bar{g}.
	\]
	This shows that $f$ is a proper strictly convex function.
	We also see that $f=\frac{\rho}{2}s^2$
	is a distance function from the unique minimum of $f$ up to a scaling. 
	It is easy to see that the radial curvatures vanish and then that
	the space $\left( \mathbb{R}^{1}\times N^{r}_1, \bar{g}, f\right)$ is 
	a Gaussian soliton. 
	This completes the proof of this theorem. 
\end{proof}

	\subsection{One of the multiplicities of two Ricci-eigenvalues $\lambda_a$ and $\lambda_\alpha$ is single}
In this subsection, we will study the case 
that one of the multiplicities of two Ricci-eigenvalues $\lambda_a$ and $\lambda_\alpha$ is single 
 in some neighborhood $U$ of $p$ in $M_A \cap \{ \nabla f \neq 0  \}$.
Without loss of generality, we assume that 
the multiplicity of Ricci-eigenvalue $\lambda_a$ is one, that means $r_1=1,~r_2=n-2\geq2$.
Then we shall give the local structure of a gradient Ricci soliton 
$\left(M^n, g, f \right) $ with harmonic Weyl curvature 
in this case according to the integrability condition.
Types {\rm (iii)} and {\rm (ii)} for $r=2$ in Theorem \ref{local} come from this subsection.

\smallskip
First of all, it follows from Theorem \ref{mulwar} that 
$(U,g)$ is isometric to a domain in 
$ L\times L_1 \times N^{n-2}_2$
with
\[
g= ds^2 + h^2_1(s)dt^2+h^2_2(s) \tilde{g}_{2}, 
\]
where $(N_2, \tilde{g}_{2})$ is an $(n-2)$-dimensional Einstein manifold 
with the Einstein constant $(n-3)k_2$.

\begin{lemma}\label{lemma5.5}
	Let $(M^n, g, f)$, $(n\geq 4)$, be an $n$-dimensional gradient Ricci soliton with harmonic Weyl curvature.
	in some neighborhood $U$ of $p$ in $M_A \cap \{ \nabla f \neq 0  \}$, 
	assume $\lambda_a$ and $\lambda_\alpha$ are mutually different Ricci eigenvalues with multiplicities $1$ and $n-2$.
	 Then 
	 \[
	 X:=\xi_a\neq0.
	 \] 
\end{lemma}

\begin{proof}
	If $X=0$, then $Y\neq0$ since $X$ and $Y$ are distinct.
	$X=0$ implies $h_1$ is constant and $Y^{'} + Y^2=0$. Thus $Y= \frac{1}{s-c_1}$ and 
	$h_2= c_{h_2} (s-c_1)$ with integrating $Y= \frac{h^{'}_2}{h_2}$.
	Putting $X=0$ into \eqref{4.3} yields
	$\rho=0$. 
	Then it follows from \eqref{5.1} that $f'=0$, which is impossible.
	The proof of this lemma is completed. 
\end{proof}	

Next, we will discuss two cases according to 
whether the constant $k_2$ vanishes, which is equivalent to that
the $(n-2)$-dimensional Einstein manifold $(N_2, \tilde{g}_{2})$ is Ricci flat.
Then the local structure of both cases will be shown
according to the integrability condition.

\medskip
\noindent{\bf Subcase I.}  \quad $k_2=0$ 

For this subcase, the following result shows that the Ricci soliton is steady and it can be classified. 
This will give rise to type {\rm (iii)} of Theorem \ref{local}.

\begin{theorem} \label{twoone1}
	For a gradient Ricci soliton $\left(M^n, g, f \right) $ with harmonic Weyl curvature,
	assume in some neighborhood $U$ of $p$ in $M_A \cap \{ \nabla f \neq 0  \}$, 
	$(M^n,g)$ is locally isometric to a domain in 
	$L\times L_1\times N^{n-2}_2$
	with
	\[
	g= ds^2 + h^2_1(s)dt^2+h^2_2(s) \tilde{g}_{2}, 
	\]
	where the $(n-2)$-dimensional Einstein manifold $(N_2, \tilde{g}_{2})$ is Ricci flat.
	  Then $n\neq5$, the gradient Ricci soliton is steady {\rm (}i.e., $\rho=0${\rm)} and $g$ is locally isometric to the metric 
	$ds^2 + s^{\frac{2(n-3)}{n-1}} dt^2+ s^{\frac{4}{n-1}} \tilde{g}$ 
	on a domain of $\mathbb{R}^+\times \mathbb{R}\times N^{n-2}_2$, where $ \tilde{g}$ is Ricci flat. 
	Also, the potential function is given by $f=\frac{2(n-3)}{n-1} \log (s)$ modulo a constant.
	
	Furthermore, the Ricci curvature components and the scalar curvature are as follows;
	$R_{11} = \frac{2(n-3)}{(n-1)s^2}$,  $R_{22} = -\frac{2(n-3)^2}{(n-1)^2s^2}$, 
	$R_{\alpha\alpha}= -\frac{4(n-3)}{(n-1)^2s^2}$, $R_{ij} =0$ $(i \neq j)$ and 
	$R= -\frac{4(n-3)^2}{(n-1)^2s^2} $. 
	Hence the scalar curvature is negative and non-constant.
\end{theorem}

\begin{proof} 
	Let $X:=\xi_a$ and $Y:=\xi_\alpha$ be the functions corresponding to the Ricci eigenvalues 
	$\lambda_a$ and $\lambda_\alpha$, respectively.
	
\noindent {\it Claim}.
	{\rm (i)} $Y\neq0$.
	
	\smallskip
\quad\quad \ {\rm (ii)} $\rho=0$.
	
	\smallskip
\quad\quad \ {\rm (iii)} $(n-3)Y-2X=0$. 
	
	In fact, in this case, we note that $r_1=1$ and $k_2=0$. Then \eqref {5.4} gives us
	$(n-3)Y-f'=0$, which shows $Y\neq0$.
Meanwhile, from (ii) of Lemma \ref{lemma5.3}, it follows that
	$\rho=0$ and 
\[
	[(n-3)Y-2X]Y= 0,
\]
	which implies (iii) since  $Y\neq0$.

	It is easy to see that $n\neq5$ from (iii) of {\it Claim} because $X$ and $Y$ are distinct.
	So
	\[
	X'+X^2=Y'+Y^2= -\frac{2}{n-3}X' +\frac{4}{(n-3)^2}X^2,
	\]
	which is 
	\[
	X'+\frac{n-1}{n-3}X^2=0. 
	\]
	Solving the above equation, 
	$X=\frac{1}{qs-c_1}$ for some constant $c_1$, where $q=\frac{n-1}{n-3}$.
	Using \eqref{4.3}, $f'=2X$; by integrating this equation,
	the potential function can be expressed as $f=\frac{2}{q} \log (s)$ modulo a constant.
	From $\frac{h'_1}{h_1}=X$ and $\frac{h'_2}{h_2}=Y$, we can get functions $h_1$ and $h_2$.
	Putting them into \eqref{4.2}, \eqref{4.3} and \eqref{4.4}, one can easily get the Ricci curvature components and the scalar curvature of $g$.  
	The proof of this theorem is completed. 	
\end{proof}

\noindent{\bf Subcase II.}  ~~~~ $k_2\neq0$ 

For this subcase, we have the following classification result,
which forms type {\rm (ii)} of Theorem \ref{local} for $r=2$.
	
\begin{theorem} \label{twoone2}
		For a gradient Ricci soliton $\left(M^n, g, f \right) $ with harmonic Weyl curvature,
	assume in some neighborhood $U$ of $p$ in $M_A \cap \{ \nabla f \neq 0  \}$, 
	$(M^n,g)$ is locally isometric to a domain in 
	$\mathbb{R}^{1}\times \mathbb{R}^{1} \times N^{n-2}_2$
	with
	\[
	g= ds^2 + h^2_1(s)dt^2+h^2_2(s) \tilde{g}_{2}, 
	\]
	where the $(n-2)$-dimensional Einstein manifold $(N_2, \tilde{g}_{2})$ is not Ricci flat.
	Then it is not steady (i.e., $\rho\neq0$).  The two distinct eigenvalues are exactly  0 and $\rho$
	with multiplicities $2$ and $(n-2)$. Also, $\nabla f$ is a null Ricci-eigenvector.
	
	Moreover, $(U,g)$ is locally isometric to a domain in 
	$\mathbb{R}^{2}\times N^{n-2}$
	with $g= ds^2 + s^2dt^2+\tilde{g} $,
	where $\left(N^{n-2}, \tilde{g}\right)$ 
	is an $(n-2)$-dimensional Einstein manifold
	with the Einstein constant $\rho\neq 0$.
	The potential function is given by $f = \frac{\rho}{2}s^2$ modulo a constant.
\end{theorem}

\begin{proof} 
First of all, we claim that 
\[
Y=0.
\]
In fact, in this case, we note that $r_1=1$ and $k_2\neq0$,
from (i) of Lemma \ref{lemma5.3}, and it follows that
	\[
	(r_1-1)X^2+(r_2-1)Y^2-2XY-\rho \neq 0.
\]
Combining with \eqref{5.3} implies $Y=0$ since $X\neq0$.

Next, the similar method of the proof of Theorem \ref{twodis} can be used to treat
this subcase. 
	$Y=0$ implies that $h_2$ is constant and $X^{'} + X^2=0$. Thus $X= \frac{1}{s-c_1}$ and 
	$h_1= c_{h_1} (s-c_1)$ with integrating $X= \frac{h^{'}_1}{h_1}$ for constants $c_1$ and $c_{h_1}$.
	By \eqref{4.2}, we have
	\[
	\lambda_1=-f''+\rho=0,
	\]
	which yields that $f''=\rho$ and then
	\[
	f(s) = \frac{1}{2} \rho (s-c_1)^2+C_1.
	\]
From \eqref{4.3} and \eqref{4.4}, we have 
	\[
		\lambda_a=-f'X+\rho=0 \quad {\rm and } \quad
	\lambda_\alpha=\rho=(n-3)\frac{k_1}{h^2_2}\neq0.
	\]
	Therefore, $(M,g)$ is locally isometric to a domain in 
	$\mathbb{R}^{2}\times N^{n-2}$
	with $g= ds^2 + s^2dt^2+\tilde{g_2}$,
	where $\left(N^{n-2}, \tilde{g_2}\right)$ 
	is an $(n-2)$-dimensional Einstein manifold
	with the Einstein constant $\rho\neq 0$.	
\end{proof}

	\section
	{The local structure of the case with the same Ricci-eigenvalues}
	
	In this section we treat the case that all Ricci-eigenvalues are equal, except for the first one.
	We set
	\[
	\lambda_{\alpha}:=\lambda_{2}= \cdots= \lambda_{n}.
	\]
	Type {\rm (iv)} of Theorem \ref{local} comes from this section.
	
	\smallskip
	For convenience, here we make the following convention on the range of indices:
	\[
	2 \leq \alpha, \beta, \gamma \cdots\leq n\]
	and denote $\xi_\alpha:=X=\frac{h'}{h}$.
	From Section 3, we have
	\begin{equation}\label{6.1}
	\frac{h''}{h}=X'+X^2=-\frac{R'}{2(n-1)f'},
	\end{equation}
	\begin{equation}\label{6.2}
	\lambda_{1}=-f''+\rho=-(n-1)\left( X'+X^2 \right)
	\end{equation}
	and
	\begin{equation}\label{6.3}
	\lambda_{\alpha}=-f'X+\rho=-\left( X'+X^2 \right)-(n-2)\left( X^2-\frac{k}{h^2} \right). 
	\end{equation}
	
	Consequently, we have the following theorem.
	\begin{theorem} \label{same}
	Let $(M^n, g,f)$ be an $n$-dimensional gradient Ricci soliton 
	with harmonic Weyl curvature. 	
	Suppose  that in some neighborhood $U$ of $p$ in $M_A \cap \{ \nabla f \neq 0  \}$, 
    all Ricci-eigenvalues are equal except for the one with respect to the gradient vector of the potential function.
	Then $g$ is a warped product:
		\begin{equation} \label{metr}
		g= ds^2 +  h^2(s) \tilde{g},
		\end{equation}
		for a positive function $h$,
		where the Riemannian metric $\tilde{g}$ is Einstein. 
		Furthermore, the $D$ tensor vanishes, so does the Bach tensor. 
	\end{theorem}

	\begin{proof}
		The first part of this theorem has been obtained by Theorem \ref{mulwar} in the third section.
		In particular, if all Ricci-eigenvalues are equal, then the metric is Einstein. 
		If $f$ is not constant, then the conclusion of Theorem \ref{same} still holds.  
		In fact,  Cheeger and Colding \cite{CC} presented a
		characterization of warped product structures on a Riemannian manifold $M$ in terms
		of solutions to the more general equation
		\[
		Hess(f)=\mu g
		\]
		for some smooth function $\mu$ on $M$. 
		More precisely, the Einstein metric $g$ becomes locally of the form 
		\[
		g= ds^2 + (f'(s))^2 \tilde{g}
		\]
		where $\tilde{g}$ is an Einstein metric. 
		
		Next we only need to verify the second part.
		It is easy to check that the gradient Ricci soliton is $D$-flat and Bach-flat. 
		In fact, from Section 3, its Riemannian curvatures are expressed as
		\[
		R_{1\alpha 1\beta}=-\left( X'+X^2\right)\delta_{\alpha\beta}=-\frac{R'}{2(n-1)}\delta_{\alpha\beta}
		\]
		and 
		\[
		R_{\alpha\beta\gamma\delta}
		=\tilde {R}_{\alpha\beta\gamma\delta}-
		X^2\left( \delta_{\alpha\gamma}\delta_{\beta\delta}-\delta_{\alpha\delta}\delta_{\beta\gamma}\right).
		\]
		
		By equations \eqref{6.2} and \eqref{6.3}, the scalar curvature and the Schouten tensor are as follows
		\[
		R=-2(n-1)\left( X'+X^2\right)-(n-1)(n-2)\left(X^2-\frac{k}{h^2} \right),
		\]
		\[
		A_{11}=-(n-2)\left( X'+X^2\right)+\frac{n-2}{2}\left(X^2-\frac{k}{h^2} \right)
		\]
		and 
		\[
		A_{\alpha\beta}=-\frac{n-2}{2}\left(X^2-\frac{k}{h^2}\right)\delta_{\alpha\beta}.
		\]
		Putting them into \eqref{2.15}, the Weyl tensor is given by
		\[
		W_{1\alpha1\beta}=0, \quad W_{1\alpha\beta\gamma}=0
		\]
		and
		\[
		W_{\alpha\beta\gamma\delta}
		=\frac{1}{h^2}\tilde {R}_{\alpha\beta\gamma\delta}-
		\frac{k}{h^2}\left( \delta_{\alpha\gamma}\delta_{\beta\delta}-\delta_{\alpha\delta}\delta_{\beta\gamma}\right).
		\]
		Finally, the harmonicity of Weyl tensor implies that the Cotton tensor $C_{ijk}=0$.
		With the relationships \eqref{2.26} and \eqref{2.25},
		it follows that $D$ tensor vanishes and the gradient Ricci soliton is Bach-flat. 
	\end{proof}
	
	\section{Classification of gradient Ricci solitons with harmonic Weyl curvature}
	In this section, we summarize and prove the theorems stated in the introduction.
Now we are going to combine theorems \ref{twodis}, 
\ref{twoone1}, \ref{twoone2} and \ref{same} to prove Theorem \ref{local}, 
in a similar way to the $4$-dimensional case of Kim \cite{Kim}.

\medskip
\noindent {\bf Proof of Theorem \ref{local}.} 
When the real analytic potential function $f$ is constant on some non-empty open subset, 
then it is constant on $M$ because of the connectivity of $M$.
So, if the soliton is of type {\rm (i)} on some non-empty open subset, it will be so on $M$.

If the soliton is of type {\rm (iv)} on some non-empty open subset with nonconstant $f$, 
then $D=0$ on $U$ and the real analytic function $\left| D \right| =0$ everywhere on $M$. 
Hence the soliton is of type {\rm (iv)} on $M$ and $f$ is nonconstant on $M$. 
The Ricci tensor of a gradient Ricci soliton with vanishing $D$-tensor 
either has a unique eigenvalue, or has two distinct eigenvalues of 
multiplicity 1 and $(n-1)$ respectively. 
Hence types {\rm (ii)} and {\rm (iii)} do not satisfy $D=0$.

For type {\rm (ii)}, the scalar curvature $R=(n-r-1)\rho$ is a nonzero constant on some non-empty open subset $U$,
and by real analyticity, $R=(n-r-1)\rho$ on $M$.
However, for type {\rm (iii)}, the scalar curvature 
$R= -\frac{4(n-3)^2}{(n-1)^2s^2}$ is not constant.
Therefore, among types {\rm (i)}-{\rm (iv)} in Theorem \ref{local}, each type is different from the other three types.
This completes the proof of Theorem \ref{local}.
\hfill $\Box$
\medskip

\vspace{0.3cm}

\medskip
	\noindent {\bf Acknowledgments}
	This work was essentially completed while the author was visiting Lehigh University from August, 2019 to August, 2020. 
	She would like to thank her advisor Professor Huai-Dong Cao  
	for suggesting the research project, and for his invaluable guidance, constant encouragement and support.
	She is grateful to her advisors Professor Yu Zheng and Professor Zhen Guo for their constant encouragement and support.
	She also would like to thank Junming Xie, Jiangtao Yu, and other members of the geometry group at Lehigh for their interest, 
    helpful discussions, and suggestions during the preparation of this paper. 
    She is particularly grateful to Professor Jongsu Kim for pointing out an error in the first (arXiv) version.
    She also would like to thank the China Scholarship Council (No: 201906140158), the Scientific Research Foundation of Chongqing University of Technology
    (No.2021ZDZ023) and the Science and Technology Research Project of Chongqing Education Commission
    (No.KJZD-K202201102).
    and the Program for Graduate Students of East China Normal University (No: YBNLTS 2020-044)
    for the financial support, and the Department of Mathematics at Lehigh University
    for hospitality and for providing a great environment for research.

	\bibliographystyle{amsplain}

\end{document}